\documentclass[12pt]{amsart}
\usepackage{}
\usepackage{amsmath,amsthm}
\usepackage{amsfonts,amssymb}

 \usepackage{color}
\usepackage{enumerate}

\newtheorem{thm}{Theorem}[section]
\newtheorem{cor}[thm]{Corollary}
\newtheorem{lem}[thm]{Lemma}

\theoremstyle{remark}

\numberwithin{equation}{section}

\def\hb{\hfill$\Box$}

\def\lb{\langle}
\def\rb{\rangle}

\def\og{\omega}

\newcommand{\thmref}[1]{Theorem~\ref{#1}}

\def\TT{\mathbb{T}}

\def\sph{\mathbb{S}^{d-1}}
\def\f{\frac}

\def\Bl{\Bigl} \def\Br{\Bigr}
 \def\la{{\langle}}
 \def\ra{{\rangle}}
 \def\bl{\bigl}
\def\br{\bigr}
\def\({\left(}
\def \){ \right)}
\def\sa{\sigma}

 \def\a{{\alpha}}

 \def\k{{\kappa}}
 
 \def\l{{\lambda}}

 \def\s{{\sigma}}
 \def\D{{\Delta}}
 \def\va{\varepsilon}

 \def\CD{{\mathcal D}}
 
 \def\CH{{\mathcal H}}

 \def\CR{{\mathcal R}}

 \def\BB{{\mathbb B}}

 \def\RR{{\mathbb R}}
  \def\SS{{\mathbb S}}
 \def\ZZ{{\mathbb Z}}
 
 \def\proj{\operatorname{proj}}

\def\leqs{\leqslant}
\newcommand{\wt}{\widetilde}

\newcommand{\p}{\partial}

\newcommand{\xa}{\la x,\a\ra}
\newcommand{\ww}{h_{\k}^2(x)}
\def\sph{\mathbb{S}^{d-1}}
\def\sphd{\mathbb{S}^{d}}

\def\Edd{\RR^d}

\def\leqs{\leqslant}

\def\ball{\mathbb{B}^d}

 \def\proj{\operatorname{proj}}

\def\HH{\mathcal{H}}
\begin{document}

\title[]{ Uncertainty Principles on weighted spheres, balls and simplexes}

\author{Han Feng}
\address{Department of Mathematical and Statistical Sciences\\
University of Alberta\\ Edmonton, Alberta T6G 2G1, Canada.}
\email{hfeng3@ualberta.ca}

\thanks{The  author  was partially supported  by the NSERC Canada
under grant RGPIN 311678-2010.
}
\maketitle

\begin{abstract}
 This paper studies the uncertainty principle for  spherical $h$-harmonic expansions on the unit sphere of $\RR^d$ associated with a weight function   invariant under a general finite reflection group, which
 is in full analogy with the classical Heisenberg inequality.  Our proof is motivated by a new decomposition of the Dunkl-Laplace-Beltrami operator on the weighted sphere.
\end{abstract}
\section{Introduction}
The uncertainty principle is a fundamental result in quantum mechanics, and it can be formulated in the Euclidean space $\Edd$, in the form of the classical Heisenberg inequality, as
\begin{equation}\label{Heisenberg}
  \inf_{a\in \Edd} \int_{\Edd}\|x-a\|^2 |f(x)|^2 dx\int_{\Edd} |\nabla f(x)|^2 dx\geq \f{d^2} 4\left(\int_{\Edd} |f(x)|^2\right)^2,
\end{equation}
where $\nabla$ is the gradient operator.
There are many papers devoted to the study of this inequality and its various generalizations, see, for instance, \cite{Folland},\cite{Rosler},\cite{DX2014}.

In particular, on the unit sphere, F.Dai and Y.Xu \cite{DX2014} established the analogue result, which states that:  if $f:\sph\to\RR$ satisfying $\int_{\sph} f(x)\, d\s(x)=0$ and $\int_{\sph}|f(x)|^2 d\s(x)=1$, then
\begin{equation}\label{unweight}
  \Bl(\min_{y\in \sph} \int_{\sph} (1-\la x, y\ra) |f(x)|^2 \, d\s(x) \Br)
\bl(\int_{\sph}|\nabla_{0} f|^2d\s(x)\br) \ge C_{ d}>0.
\end{equation}

In a recent paper \cite{X_U} ,  with a weight function $h^2_\k(x)$ invariant under a group $G$, he studied  the uncertainty principle on the unit sphere $\sph$. By introducing a weighted analogue $\nabla_{\k,0}$ of the tangential gradient $\nabla_0$,
  he  proved \cite[Theorem 4.1]{X_U}  that
if $f:\sph\to\RR$ is invariant under the group $G$ and satisfies that $\int_{\sph} f(x) h_\k^2(x)\, d\s(x)=0$ and $\int_{\sph}|f(x)|^2 h_\k^2(x)d\s(x)=1$, then
\begin{equation}\label{9-1-1}
\Bl(\min_{1\leq i\leq d} \int_{\sph} (1-\la x, e_i\ra) |f(x)|^2 h_\k^2(x)\, d\s(x) \Br)
\bl(\int_{\sph}|\nabla_{\k,0} f|^2 h^2_\k(x)d\s(x)\br)\ge C_{\k, d}>0.\end{equation}
where $e_i$, $i=1,\cdots,d$, is the standard vector,namely only the $i$th coordinate is nonzero 1,  and $C_{\k,d}$ is a constant only depends on parameter $\k,d$, and $\la \cdot,\cdot\ra$ is the inner product in $\Edd$.

   The purpose of the present paper is
   to show that  the inequality \eqref{9-1-1} with minimum being taken over all $y\in \sph$ rather than the finite subset $\{e_1, \cdots, e_d\}$ remains true without the extra assumption that $f$ is $G$-invariant.

Recall that the geodesic distance on the sphere is defined by $d(x,y)=\arccos\la x,y\ra$, so that
\[1-\la x,y\ra=2\sin^2 \f{d(x,y)}2\sim d(x,y)^2\]
with $A\sim B$ meaning $\f 1 c A\leq  B\leq c A$ for some $c>0$. It implies that \eqref{unweight},\eqref{9-1-1} can be regarded as a close analogy of \eqref{Heisenberg}.

Let $G\subset O(d)$ be a finite reflection group on $\RR^d$.
For   $v\in \RR^d\setminus\{0\}$, we denote by $\sa_v$ the
reflection with respect to the hyperplane perpendicular to $v$; that is,
 $$\sa_v x= x- \f{2\la x, v\ra}{\|
v\|^2}v,\   \   \    x\in\RR^d,$$ where  $\la\cdot,
\cdot\ra$  denotes  the  Euclidean inner product on $\RR^d$ and
 $\|x\|:=\sqrt{\la x, x\ra}$.
Let $\CR$ be the  root system of $G$, normalized so that
$\la v, v\ra =2$ for all $v\in \CR$, and fix a positive subsystem $\CR_+$ of $\CR$, such that $\CR=\CR_+\cup (-\CR_+)$. From the general theory of reflection groups (see, e.g.,\cite{DX} ), the set of reflections in $G$  associates with  $\{\s_v:\  \ v\in \CR_+\}$, which also generates the group $G$.
 Let  $\k:\ \CR\to [0,\infty)$, $v\mapsto \k_v=\k(v)$ be a nonnegative  multiplicative function on $\CR$; that is, $\k$ is a nonnegative $G$-invariant function on $\CR$. Let $h_\k$ denote the weight function on $\RR^d$ defined by
\begin{equation}\label{1-1-t}h_\kappa(x):=\prod_{v\in \CR_+} |\la x, v\ra
|^{\k_v},\    \   \    x\in \RR^{d}. \end{equation}
It is $G$-invariant and homogeneous of degree $|\k|:=\sum_{v\in \CR_+}\k_v$.

Let $\Delta_{\k,0}$ be  the weighted analogy of the Laplace-Beltrami operator $\D_0$ on $\sph$, whose precise definition will be given  in next section.
Then our main result can be stated as follows:

\begin{thm}\label{thm-9-1}
Let $f\in C^1(\sph)$ be such that $\int_{\sph} f(x) h_\k^2(x)\, d\s(x) =0$ and $\int_{\sph} |f(x)|^2 h_\k^2(x)\, d\s(x) =1$.
Then
\begin{align}
\Bl[&\min_{y\in\sph} \int_{\sph}(1-\la x, y\ra ) |f(x)|^2
h_\k^2(x)\, d\s(x) \Br]\times\notag\\
&\times  \Bl[ \int_{\sph} | \sqrt{-\Delta_{\k,0}}
f(x)|^2 h_\k^2(x)\, d\s(x)\Br]\ge C_{d,\k}>0, \label{9-1:uncertainty}\end{align}
where $C_{d,\k}$ is a constant depending on $d$ and $\k$ only.
\end{thm}

As a direct corollary , we obtain the following improvement of Theorem 4.1 and Theorem 4.2 of  \cite{X_U}:
\begin{cor}\label{cor-9-2}If $f\in C^1(\sph)$  satisfies that $\int_{\sph} f(x) h_\k^2(x)\, d\s(x)=0$ and $\int_{\sph} |f(x)|^2 h_\k^2(x)\, d\s(x) =1$, then
\begin{equation}
\Bl(\min_{y\in \sph} \int_{\sph} (1-\la x, y\ra) |f(x)|^2 h_\k^2(x)\, d\s(x) \Br)
\bl(\int_{\sph}|\nabla_{\k,0} f|^2 h^2_\k d\s(x)\br) \ge C_{\k, d}>0.\end{equation}
\end{cor}
Noteworthy, the improvement by taking minimum over all $y\in \sph$ instead of $\{e_1, \cdots, e_d\}$ is nontrivial since the weight $h^2_\k$ is not invariant under all rotations. And obviously, the requirement of the $G$-invariance of $f$ turns out to be not necessary.

Finally, we shall also  establish similar results
  for the weighted orthogonal polynomial
expansions (WOPEs)
 with respect
to the weight function
\begin{equation} \label{weightB}
 W_\k^B(x): = \Bl(\prod_{v\in \CR_+} |\la x, v\ra|^{2\k_v}\Br)(1-\|x\|^2)^{\mu-1/2},\   \
 \qquad \mu \ge 0
\end{equation}
on the unit ball $\BB^d := \{x\in\RR^d: \|x\| \le 1\}$ where $\CR_+$, $\k$  are adopted as before, as well as
 for the WOPEs  with respect to  the weight
 function
\begin{equation}  \label{weightT1}
 W_\k^T(x; \ZZ^d_2) := \Bl(\prod_{i=1}^{d} x_i ^{\k_i-1/2}\Br)(1-|x|)^{\k_{d+1}-1/2},
      \qquad \min_{1\leq i\leq d+1}\k_i \ge 0,
\end{equation}
or
 \begin{equation} \label{weightT2}
   W_{\k,\mu}^T(x;H_d)=\prod_{i=1}^d x_i^{\k'-1/2}\prod_{1\leqs i<j\leqs d}|x_i-x_j|^\k(1-|x|)^{\mu-1/2}, \qquad \min\{\k',\k,\mu\} \ge 0,
 \end{equation}
on the simplex
 $\TT^d :=\{x\in\RR^d:x_j\ge 0, \ldots, x_d\ge 0, 1-|x| \ge 0\}$, here, and in what follows,
$|x|:=\sum_{j=1}^d |x_j|$ for $x=(x_1,\cdots, x_d)\in\RR^d$.

\section{Preliminaries}

\subsection{The Dunkl Theory}

This theory of spherical
$h$-harmonics was  initially developed by C.F. Dunkl in \cite{Dunkl1, D1, D2}.  For the details, one can refer to for instance \cite{DX} and \cite{WHP}.
Let $\CR$ be a fixed root system in $\RR^d$ normalized so that $\la v, v\ra =2$ for all $v\in \CR$,  and $G$ the associated reflection group.
Let $\k: \CR\to [0,\infty)$ be  a multiplicity function on $\CR$.

 The Dunkl operators associated with
$G$ and $\k$ are defined  by
\begin{equation*}\label{2-1-eq}\mathcal{D}_{ i} f (x) =\p_i f(x) +
\sum_{v\in \CR_+} \k_v  \la v, e_i\ra \f{f(x)-f(\s_v x)}{\la x,v\ra},\   \ i=1,\cdots, d, \ \  f\in C^1(\RR^d),\end{equation*} where
$\p_i=\f{\p}{\p x_i}$, $\CR_+$ is a fixed positive subsystem of $\CR$.
Here we use the notation   $g\circ f (x) :=f(gx)$ for $g\in G$, $f\in C(\sph)$ and $x\in\sph$.

The $\k$-Laplacian on $\RR^d$ is defined by
$\Delta_\k :=\sum_{j=1}^d \mathcal{D}_j^2.$
The operator $\Delta_\k$ is $G$-invariant; that is,  $g\circ \Delta_\k = \Delta_\k\circ g$ for all $g\in G$. Similarly, the $\k$-grandient is defined by $\nabla_\k=(\CD_1,\cdots,\CD_d)$. Furthermore, by restricting on the unit sphere,  the weighted analogue $\D_{\k,0}$ of Laplace-Beltrami operator $\D_{0}$ and analogue $\nabla_{\k,0}$ of the tangential gradient $\nabla_0$  are defined as follows:
  \begin{equation*}\label{2-16:sec2}
   \Delta_{\k,0}f(x):=\Delta_\k F(z)|_{z=x}, \qquad \forall\, x\in \sph
\end{equation*}
and
  \begin{equation*}\label{2-16:sec3}
   \nabla_{\k,0}f(x):=\nabla_\k F(z)|_{z=x}, \qquad \forall\, x\in \sph
\end{equation*}
where $F(z)=f(\f z{\|z\|})$.

\subsection{$h$-harmonic expansions}

Let $\sph =\{x\in\RR^d: \|x\| =1\}$ denote  the unit sphere of
$\RR^{d}$ equipped  with the usual Haar measure $d\sa(x)$, and the weight function  $h_\k$  given in \eqref{1-1-t}.  For $1<p<\infty$, recall that
$$
   \|f\|_{\k,p} :=
  \Big( \int_{\sph} |f(y)|^p h_\kappa^2(y) d\sa(y)
\Big)^{1/p}.
$$
We denote by $\Pi_n^d$ the space of all spherical polynomials of degree at most $n$ on $\sph$, and $\CH_n^{d}(h_\k^2)$
 the space of all spherical $h$-harmonics of degree $n$ on $\sph$. Thus, $\CH_n^{d}(h_\k^2)$ is the orthogonal complement of $\Pi_{n-1}^d$ in the space $\Pi_n^d$ with respect to the inner product
 $$\la f, g\ra_\k :=\int_{\sph} f(x) \overline{g(x)} h_\k^2(x)\, d\s(x),$$
 and  each function  $f\in L^2(h_\k^2; \sph)$ has
   a spherical $h$-harmonic   expansion
$f = \sum_{n=0}^\infty \proj_n(h_\k^2; f)$
converging
  in the norm of $L^2(h_\k^2; \sph)$.

  Here $\proj_n (h_\k^2): L^2(h_\k^2; \sph)\to \HH_n^d(h_\k^2)$ is the orthogonal projection. Also, the projection $\proj_n (h_\k^2;  f)$ can be extended
to all $f\in L^1(h_\k^2; \sph)$ in the sense that
\[\proj_n (h_\k^2;  f)(x)=\int_{\sph} f(y)P_n(x,y)h^2_\k(y)d\s(y), \   f\in L^1(h_\k^2; \sph),\]
with $P_n(h^2_\k; x,y)$ being the reproducing kernel of $\HH_n^d(h_\k^2)$.

  A crucial   point  in the theory of $h$-harmonics is that the space $\HH_n^d(h_\k^2)$   can also be seen   as an   eigenspace of a second order differential-difference operator $\Delta_{\k,0}$ corresponding to the eigenvalue $-n (n+2\l_\k)$ . Here and throughout the paper, \begin{equation*}\label{1-3-0}
   \l_\k: =\f {d-2} 2+|\k|.\end{equation*}

   Given $\a\in \RR$, we define the fractional power  $(-\Delta_{\k,0})^\a$ of $(-\Delta_{\k,0})$, in a distributional sense,  by
  \begin{equation*}\label{1-1-20}
   \proj_n(h_\k^2; (-\Delta_{\k,0})^\a f)   = (n(n+2\l_\k))^{\a}\proj_n(h_\k^2; f),\   \  n=0, 1,\cdots.
  \end{equation*}

 Next we introduce a first order differential operator on suitable functions defined on $\Edd$
\begin{equation*}
  D_{i,j}f(x)=x_j\p_i f(x)-x_i\p_jf(x), \  \  1\leq i,j\leq d
\end{equation*}
and
\begin{equation*}\label{1-11-E}
E_v f(x)= \f {f(x)-f(\s_v x)}{\la x,v \ra},\  \ v\in\RR^d\setminus\{0\}.
\end{equation*}
The proof of our main result relies on a decomposition of $(-\D_{\k,0})$ and an practical estimate of $\|(-\Delta_{\k,0})^{1/2} f \|_{\k,p}$ in \cite{DF}, which is stated as the following theorem.
\begin{thm}\cite{DF}\label{thm-1-5}
For $f\in C^1(\sph)$, with the notation given above,
  \begin{equation}\label{1-2-2}
  \| (-\Delta_{\k,0})^{1/2}f\|_{\k,2}^2 =\sum_{1\leq i<j\leq d}\| D_{i,j} f\|_{\k,2}^2+\sum_{v\in \CR_+}\k_v\| E_v f\|_{\k,2}^2 .
  \end{equation}
  \end{thm}

  Particularly, in the unweighted setting, namely when $\k=0$, this theorem will go back to the classical result (see for instance \cite[Section 1.8]{DaXuBook}) that for $f\in C^1(\sph)$,
\begin{equation}\label{UPP1}
  \| (-\Delta_{0})^{1/2}f\|_{2}^2=\|\nabla_0 f\|_2^2 =\sum_{1\leq i<j\leq d}\| D_{i,j} f\|_{2}^2.
  \end{equation}
  where $\|g\|_2^2=\int_{\sph}|g(x)|^2d\s(x)$, $g\in L^2(\sph)$.

\section{The proof of Corollary \ref{cor-9-2}}
For the moment, we take Theorem \ref{thm-9-1} for granted and proceed with the proof of Corollary \ref{cor-9-2}.
\begin{proof} By \eqref{9-1:uncertainty},    it suffices to show
\begin{equation}\label{9-4:uncertainty}
\|\sqrt{-\Delta_{\k,0}} f \|_{\k,2} \leq  \|\nabla_{\k,0} f\|_{\k,2}.
\end{equation}
Indeed, noticing (3.15), (3.13) of \cite{X_U}, we have that
 \begin{equation}\label{9-5}
\|\sqrt{-\Delta_{\k,0}} f\|_{\k,2}^2=\|\nabla_{h,0} f \|_{\k,2}^2   -\f{2\l_\k}{\og_d^\k} \int_{\sph} ( \xi\cdot \nabla_{h,0} f(\xi)) f(\xi) h_\k^2(\xi)\, d\s(\xi),
\end{equation}
where $\og_d^\k =\int_{\sph} h_\k^2(x)\, d\s(x)$.
 Here it should be pointed out that the last two terms in (3.15) of \cite{X_U} in fact can be cancelled out by realising that
\[(I-\s_v)^2=2(I-\s_v), \  \  \forall v\in \CR_+.\]

Furthermore,  by
(3.3) of \cite{X_U}, we obtain
\begin{align*}
\int_{\sph} ( \xi\cdot \nabla_{h,0} f(\xi)) f(\xi) h_\k^2(\xi)\, d\s(\xi)=\sum_{v\in \CR_+} \k_v \int_{\sph} ( f(\xi)-f(\s_v \xi))f(\xi)h_\k^2(\xi)\, d\s(\xi).
\end{align*}
However, by the Cauchy-Schwartz  inequality,
$$ \int_{\sph} f(x) f(\s_vx) h_\k^2(x)\, d\s(x) \leq \|f\|_{\k,2}^2,\   \  \forall v\in \CR_+.$$
Thus,
\begin{equation*}\int_{\sph} \Bl(\xi\cdot \nabla_{h,0}) f(\xi)\Br)  f(\xi) h_\k^2(\xi) \, d\s(\xi)\ge 0.\end{equation*}
The desired inequality \eqref{9-4:uncertainty}  then follows by \eqref{9-5}.
\end{proof}

\section{The Proof of \thmref{thm-9-1}}
Now we turn to the proof of Theorem \ref{thm-9-1}.
Recall that $\l_\k=\f{d-2}2+|\k|$ and $|\k|=\sum_{\a\in \CR_+}\k_\a$.

 Our proof crucially relies on the following lemma.
 \begin{lem} If $f\in C^1(\sph)$ and $y\in\sph$, then
 \begin{align}
   \Bl (\f{d-1}2+|\k|\Br)& \int_{\sph} \la x, y\ra
    |f(x)|^2 h_{\k}^2 (x)\, d\s(x)
    =\sum_{\a\in \CR_+} \k_{\a} \lb y,\a\rb \int_{\sph}  \f{|f(x)|^2\ww}{\xa}\,d\s(x)\notag\\
      &-\int_{\sph}
    \Bl[\sum_{i=1}^d \sum_{j=1}^d   x_j y_i D_{i,j} f(x) \Br] f(x) h_{\k}^2 (x) \, d\s(x), \label{9-4-0}\end{align}
    where $x_j=\la x, e_j\ra$ and $y_j=\la y, e_j\ra$.
\end{lem}
\begin{proof}
By noticing that
for $f, g\in C^1(\sph)$ and $i\neq j$,
    \begin{equation*}\label{2-21:sec2}
\int_{\sph}f(x)D_{i,j}g(x)d\s(x)=-\int_{\sph} D_{i,j}f(x)g(x)d\s(x),
\end{equation*}
 we obtain that  for $2\leq j\leq d$,
\begin{align*}
    \int_{\sph} \Bl[ x_j D_{1,j} f(x)\Br] f(x) h_{\k}^2(x)\, d\s(x) &=-\int_{\sph} f(x) \Bl[D_{1,j} f(x)\Br] x_j h_{\k}^2(x)\, d\s(x) \\
    &-\int_{\sph} |f(x)|^2 \Bl[D_{1,j} \bl( x_j h_{\k}^2 (x)\br)\Br]\, d\s(x).
\end{align*}
A  straightforward calculation shows that
\begin{align*}
  &D_{1,j}\bl(x_jh_{\k}^2(x)\br)
  =\Bl(x_1+x_1\sum_{\a\in \CR_+}\f{2\k_{\a}x_j\a_j}{\xa}-x_j^2\sum_{\a\in \CR_+}\f{2\k_{\a}\a_1}{\xa}\Br)h_\k^2(x),
\end{align*}
where $\a_j=\la \a, e_j\ra$.
Thus,
\begin{align*}
   2 \int_{\sph} \Bl[ x_j D_{1,j} f(x)\Br] f(x) h_{\k}^2(x)\, d\s(x)
   =&\int_{\sph} |f(x)|^2  x_j^2\Bl(\sum_{\a\in \CR_+}\f{2\k_{\a}\a_1}{\xa} \Br)h_{\k}^2(x)\, d\s(x)\\
   &-\int_{\sph} |f(x)|^2\Bl[x_1+x_1\sum_{\a\in \CR_+}\f{2\k_{\a}x_j\a_j}{\xa}\Br] h_\k^2(x)\,d\s(x)
   \end{align*}
   Summing this last equation  over $j=2,\cdots, d$ yields
   \begin{align*}
    \int_{\sph}&
    \Bl[ \sum_{j=2}^d x_j
    D_{1,j} f(x)\Br] f(x) h_{\k}^2(x)\, d\s(x)
    =\int_{\sph}|f(x)|^2\sum_{\a\in \CR_+}\f{\k_{\a}\a_1}{\xa}\ww\,d\s(x)\\
   & -\Bl(|\k|+\f{d-1}{2}\Br)\int_{\sph}x_1 |f(x)|^2 h_\k^2(x) \,d\s(x).
   \end{align*}

   In general,  for $1\leq i\leq d$,    recalling $D_{i,i}=0$, and using symmetry,
    we obtain
      \begin{align}
    \int_{\sph}& \Bl[ \sum_{j=1}^d
     x_j D_{i,j} f(x)\Br] f(x) h_{\k}^2(x)\, d\s(x) =\int_{\sph}|f(x)|^2\sum_{\a\in \CR_+}\f{\k_{\a}\a_i}{\xa}\ww\,d\s(x)\notag\\
     &-\Bl(|\k|+\f{d-1}{2}\Br)\int_{\sph}x_i |f(x)|^2\ww \,d\s(x)
    d\s(x).\label{9-3-1}
   \end{align}
   Multiplying both sides of \eqref{9-3-1}  by $y_i$ and
   summing the resulting equation  over $i=1,\cdots, d$  yield the desired identity \eqref{9-4-0}.\end{proof}

We are now in a position to prove Theorem \ref{thm-9-1} .\\

{\it Proof of Theorem \ref{thm-9-1}.}  \   \    Let
$\va\in (0,1)$  be  a small absolute constant  to  be  specified later.
 If $$\int_{\sph} \la x, y\ra |f(x)|^2  h_\k^2(x)\, d\s(x)\leq
 1-\va,$$  then
$$\int_{\sph} |f(x)|^2 (1-\la x, y\ra )h_\k^2(x)\, d\s(x)
 \ge \va,$$
 and   \eqref{9-1:uncertainty} holds trivially as
 $\|\sqrt{-\Delta_{\k,0}} f\|_{\k,2}\ge \|f\|_{\k,2}=1.$
  Thus, without loss of generality,   we  may assume that
\begin{equation}\label{9-4-1}\int_{\sph} \la x, y\ra  |f(x)|^2 h_\k^2(x)\, d\s(x)> 1-\va.\end{equation}

    We will use the identity \eqref{9-4-0}.
   Indeed, it will be shown that
    \begin{align}
       &J_1:= \Bl|\int_{\sph}
    \Bl[\sum_{i=1}^d \sum_{j=1}^d  y_i x_j D_{i,j} f(x) \Br] f(x) h_{\k}^2 (x) \, d\s(x)\Br|\notag\\
    \leq& C \|\nabla_0 f \|_{\k,2}
    \Bl(\int_{\sph} |f(x)|^2 (1-\la x, y\ra) h_{\k}^2(x)\, dx \Br)^{\f12}\label{9-5-0}
    \end{align}
   and that   for each $\a\in \CR_+$ with $\k_\a>0$,
   \begin{align}
  J_2(\a):&= \Bl| \lb y,\a\rb \int_{\sph}   \f{ |f(x)|^2 \ww}{\xa}\,d\s(x)\Br|\notag\\
  &\leq \f 1{1-\va} +\f{C}{\va}  \|E_\a f\|_{\k,2} \Bl(
    \int_{\sph} |f(x)|^2 (1-\la x, y\ra) h_{\k}^2(x) \,
    d\s(x) \Br)^{\f12}.\label{9-6-1}
   \end{align}
   Once \eqref{9-5-0} and \eqref{9-6-1} are proven, then using \eqref{9-4-0}, \eqref{9-4-1}  and \eqref{1-2-2}, we obtain
    \begin{align*}
(1-\va) \Bl(|\k|+\f {d-1} 2\Br)&\leq \f {C|\k|} \va  \|\sqrt{-\Delta_{\k,0}}
f\|_{\k,2} \Bl(
    \int_{\sph}
    |f(x)|^2 (1-\la x, y\ra) h_{\k}^2(x) \,
     d\s(x) \Br)^{\f12}\\
     & + \f {|\k|}{1-\va} .
\end{align*}
Thus,  choosing $\va\in (0,1)$ small
enough so that
$$(1-\va) \Bl(|\k|+\f {d-1} 2\Br)-\f 1{1-\va} |\k|\ge C_{d,\k}>0,$$
   we deduce the desired inequality  \eqref{9-1:uncertainty}.

    It remains to show \eqref{9-5-0} and \eqref{9-6-1}.
    For the proof of \eqref{9-5-0}, we first note that for $x\in\sph$,
    \begin{align*}
        \sum_{i=1}^d \sum_{j=1}^d x_i x_j D_{i,j}=\sum_{i=1}^d \sum_{j=1}^d  (x_i^2 x_j \p_j-x_i x_j^2 \p_i)=0.
    \end{align*}
    Thus,
    \begin{align}
        J_1
    &=\Bl|\int_{\sph}
    \Bl[\sum_{i=1}^d \sum_{j=1}^d  (y_i-x_i) x_j D_{i,j} f(x) \Br] f(x) h_{\k}^2 (x) \, d\s(x)
    \Br|\notag\\
    &\leq \Bl( \int_{\sph}
    \f{|\sum_{i,j=1}^d  (y_i-x_i) x_j D_{i,j}
    f(x)|^2}{1-\la x, y\ra} h_{\k}^2 (x) \, d\s(x) \Br)^{\f12}\times \notag\\
    &\   \   \  \times \Bl(\int_{\sph} |f(x)|^2 (1-\la x, y\ra) h_{\k}^2(x)\,
    d\s(x) \Br)^{\f12}.\notag
    \end{align}
    But,  by the Cauchy-Schwartz inequality,
    \begin{align*}
       & \Bl|\sum_{i=1}^d \sum_{j=1}^d (y_i-x_i) x_j D_{i,j} f(x)\Br|^2 \leq \Bl[ \sum_{i,j
       =1}^d |x_j|^2 (y_i-x_i)^2\Br] \Bl[\sum_{i, j=1}^d |D_{i,j}f(x)|^2\Br]\\
       &=4(1-\la x, y\ra)  \Bl[\sum_{1\leq i<j\leq d}|D_{i,j}f(x)|^2\Br]
    \end{align*}
    It follows that
  \begin{align*}
   J_1 \leq& 2 \Bl(\sum_{1\leq i<j\leq d} \int_{\sph} |D_{i,j}f(x)|^2 h_{\k}^2 (x)\,
     d\s(x)\Br)^{\f12}\Bl(\int_{\sph} |f(x)|^2
      (1-\la x, y\ra) h_{\k}^2(x)\, d\s(x) \Br)^{\f12},\end{align*}
      which  implies  \eqref{9-5-0} by \eqref{UPP1}.

    Finally, we prove \eqref{9-6-1}.
    Splitting the integral  $\int_{\sph}\cdots$ into two parts, we get \begin{equation}\label{9-7-0}
       J_2(\a)\leq  J_{2,1}(\a)+J_{2,2}(\a),
       \end{equation}
       where
        \begin{align*}
        J_{2,1}(\a):=&\Bl|\lb y,\a\rb\int_{|\xa|>(1-\varepsilon)|\lb y,\a\rb|}\f{|f(x)|^2\ww}{\xa}\,d\s(x)\Br|,\\
                 J_{2,2}(\a):=&\Bl|\lb y,\a\rb
                 \int_{|\xa|\leqslant(1-\varepsilon)|\lb y,\a\rb|}\f{|f(x)|^2\ww}{\xa}\,d\s(x)\Br|.\end{align*}
A straightforward calculation shows that
\begin{align}
    J_{2,1}(\a)
    &\leq\f 1{1-\va}
    \int_{\sph}
    |f(x)|^2 h_{\k}^2 (x) \, d\s(x)=\f1{1-\va} .\label{9-8-0}
     \end{align}

    To estimate the term $J_{2,2}(\a)$,  we first note that for any $t\in (0,1)$ and $\a\in \CR_+$,
    $$\int_ {|\xa|\leqslant
    t}\f{|f(x)|^2}{\la x, \a\ra}  h_\k^2(x)\, d\s(x)=
    \int_{|\xa|\leqslant
    t}\Bl(E_{\a}f(x)\Br) f(x)\ww\,d\s(x).$$
    Thus,
    \begin{align}
  J_{2,2}(\a)&=
     \Bl||\la y, \a\ra \int_{|\xa|\leqslant
    (1-\varepsilon)|\lb y,\a\rb|}\Bl( E_{\a}f(x)\Br)f(x)    \ww\,d\s(x)\Br|\notag\\
    &\leq \f 1{\va}  \Bl| \int_{\sph}\|x-y\|\Bl( E_{\a}f(x)\Br)f(x)    \ww\,d\s(x)\Br|\notag\\
    &\leq \f{\sqrt{2}}{\va}  \|E_\a f\|_{\k,2} \Bl(
    \int_{\sph} |f(x)|^2 (1-\la x, y\ra) h_{\k}^2(x) \,
    d\s(x) \Br)^{\f12}, \label{9-9-0}\end{align}
    where the second step uses the fact that    if
     $|\xa|\leqslant(1-\varepsilon)|\lb y,\a\rb|$,  then
    \[\varepsilon|\lb y,\a\rb|\leqslant |\lb y,\a\rb|-|\xa|
    \leqslant  \|x-y\|.\]
    Now a combination of  \eqref{9-7-0}, \eqref{9-8-0}  and \eqref{9-9-0} yields the estimate \eqref{9-6-1}.

    This completes the proof of Theorem \ref{thm-9-1}.\hb
\section{Uncertainty Principle on The unit Ball and The simplex}

In this section, we will drive uncertainty principles for weighted orthogonal polynomial expansions  on the unit ball and the simplex from results established in  the last section.

 Our argument is based on a close relation among analysis on the unit sphere, the unit ball and simplex (see, e.g. \cite{DX}, \cite[Sections 9,10]{DF}). More precisely, given two changes of  variables $y=\phi(x)$, $z=\psi(x)$ with
 \begin{eqnarray*}
 &\phi: \  \BB^d\to \SS^d,\   \
 x\in \BB^d \mapsto (x,\sqrt{1-\|x\|^2}) \in \SS^{d},\\
 &\psi: \  \BB^d \to \TT^d, \  \
 x\in \BB^d \mapsto (x_1^2,x_2^2,\cdots,x_d^d)\in \TT^d,
\end{eqnarray*}
 we have that
\begin{align}\label{BSintegral}
&\int_{\SS^d} f(y) d\s(y)\\
 = &\int_{\BB^d} \left[
f(x,\sqrt{1-\|x\|^2}\,)+
     f(x,-\sqrt{1-\|x\|^2}\,) \right] \, \f{dx}{\sqrt{1-\|x\|^2}}\notag
\end{align}
and
\begin{equation} \label{T-B}
\int_{\BB^d} g\bl( \psi(x)\br) dx = \int_{\TT^d} g(z)\,\f{ dz}{|z_1\cdots z_d|}.
            \end{equation}

 Recall that $G$ is a finite  reflection group on $\RR^d$ with a root system $\CR\subset \RR^d$;  $\k: \CR\to [0,\infty)$ is a  nonnegative multiplicity  function on $\CR$; the weight functions $W_{\k,\mu}^B$ on  $\BB^d$  and $W_{\k,\mu}^T$ on $\TT^d$ are  given in \eqref{weightB} and \eqref{weightT1}, \eqref{weightT2},  respectively.

Let $\Delta_{\k, \mu}^B$  and $\Delta_{\k, \mu}^T$  be the analogues of the Dunkl-Laplace-Beltrami operator $\Delta_{\k,0}$ on $\BB^d$ and  $\TT^d$, respectively. They are second order differential-difference operators and their precise definitions can be found in \cite[Sections 8.1,8.2]{DX}. Here we just emphasize the relations among the three operators. First,  for a function $f$ on $\BB^d$, the identity
\begin{equation}\label{9-1-5}
       (-\D_{\k,\mu}^B)^{\a}f(x)=
       (-\D_{\tilde{\k},0})^{\a}\wt{f}(\phi(x)) , \quad x\in\ball,\   \  \a\in\RR,
    \end{equation}
    holds in a distributional  sense,
where the weight associated to $\D_{\tilde{\k},0}$ is
\[h_{\tilde{\k}}(x)=|x_{d+1}|^\mu\prod_{v\in \RR_+} |\la x,v\ra|^{\k_v},\quad  x\in \sphd\]
 and $\tilde{f}(x,x_{d+1})=f(x)$.
  Second, for a function $f$ on $C^2(\TT^d)$,
\begin{equation}\label{10-4}
   \bl((-\D_{\k, \mu}^T)^\a f\br) \circ \psi(x)=4^{-\a} (-\D_{\k,\mu}^B)^\a (f\circ \psi)(x),\quad x\in \ball,\ \  \a\in\RR.
 \end{equation}

 Then the following  results on the unit ball and simplex, which are similar to that of  \thmref{thm-9-1} on the sphere, are immediate consequences of \eqref{BSintegral} ,\eqref{9-1-5} and \eqref{T-B},\eqref{10-4}:
\begin{thm}\label{ball-up}
Let $f\in C^1(\ball)$ be such that $\int_{\ball} f(x) W_{\k,\mu}^B(x)\, dx=0$ and $\int_{\ball} |f(x)|^2 W^B_{\k,\mu}(x)dx=1$.
Then
\begin{align}
\Bl[&\min_{y\in\ball} \int_{\ball}(1-\la x, y\ra ) |f(x)|^2
W^B_{\k,\mu}(x)\, d(x) \Br]\times\notag\\
&\times  \Bl[ \int_{\ball} | \sqrt{-\Delta^B_{\k,\mu}}
f(x)|^2 W^B_{\k,\mu}(x)\, d(x)\Br]\ge C_{d,\k,\mu}>0. \label{9-2:uncertainty}
\end{align}
\end{thm}

\begin{thm}\label{simplex-up}
  Let $f\in C^1(\TT^d)$ be such that $\int_{\TT^d} f(x) W_{\k,\mu}^T(x)\, dx=0$ and $\int_{\TT^d} |f(x)|^2 W^T_{\k,\mu}(x)dx=1$.
Then
\begin{align}
\Bl[&\min_{y\in\TT^d} \int_{\TT^d}(1-\la \psi^{-1}(x), \psi^{-1}(y)\ra ) |f(x)|^2
W^T_{\k,\mu}(x)\, d(x) \Br]\times\notag\\
&\times  \Bl[ \int_{\TT^d} | \sqrt{-\Delta^T_{\k,\mu}}
f(x)|^2 W^T_{\k,\mu}(x)\, d(x)\Br]\ge C_{d,\k,\mu}>0,\label{9-3:uncertainty}
\end{align}
where we recall that $\psi^{-1}(x)=(\sqrt{x_1},\sqrt{x_2}, \cdots, \sqrt{x_d})$.
\end{thm}

We remark that    Theorem \ref{ball-up} and Theorem \ref{simplex-up} above  improve the corresponding results obtained recently in \cite{X_U}. In fact,
Theorem 5.1, Theorem 5.2, Theorem 6.1, Theorem  6.2  and Corollary 5.3 of the paper \cite{X_U} follow directly from the above two theorems. We also note that equivalently, we can take  the  minimums  in the above two theorems over the sphere  $\sphd$ rather than
 the ball $\ball$.

\end{document}